\documentclass[12pt,fullpage]{article}
\usepackage{amsmath,amsfonts,amsthm,amssymb,amsopn,amstext,amscd}
\usepackage{enumerate}
\usepackage{graphicx}
\usepackage{epsfig}
\usepackage{float}
\usepackage[font={stretch=1}]{caption}
\usepackage{setspace}
\usepackage{enumitem}

\newcommand{\eps}{\varepsilon}
\newtheorem{thm}{Theorem}
\newtheorem{prop}[thm]{Proposition}
\newtheorem{lem}[thm]{Lemma}

\begin{document}

\pagenumbering{arabic}

\title{Defective Galton-Watson processes 
}
\author{Serik Sagitov\footnote{Department of Mathematical Sciences, Chalmers University of Technology and University of Gothenburg, Gothenburg, Sweden.}\hspace*{1ex} and
Carmen Minuesa\footnote{Department of Mathematics, University of Extremadura, Badajoz, Spain. E-mail address: cminuesaa@unex.es}}
\date{}
\maketitle

\begin{abstract}
 The Galton-Watson process is a Markov chain modeling the population size of independently reproducing particles giving birth to $k$ offspring with probability $p_k$, $k\ge0$. In this paper we consider {\it defective} Galton-Watson processes having defective reproduction laws, so that  $\sum_{k\ge0}p_k=1-\eps$ for some $\eps\in(0,1)$. In this setting, each particle may send the process to a graveyard state $\Delta$ with probability $\eps$.
Such a Markov chain, having an enhanced state space $\{0,1,\ldots\}\cup\{\Delta\}$, gets eventually absorbed either at $0$ or at $\Delta$.
Assuming that the process has avoided absorption until the observation time $t$, we are interested in its trajectories as $t\to\infty$
and $\eps\to0$.
\end{abstract}

\textbf{Keywords}: branching process; defective distribution; Galton-Watson process with killing; conditional limit theorems.

\textbf{MSC}: 60J80.

\section{Introduction}

The classical Galton-Watson process (GW-process) is a discrete time Markov chain  $Z(\cdot)$ with the state space $\{0,1,\ldots\}$ defined recursively by
\begin{equation}\label{def:model}
Z(0)=1,\quad Z(t+1)=\sum_{j=1}^{Z(t)}\nu_{t,j},\quad t=0,1,\ldots,
\end{equation}
where $\nu_{t,j}\stackrel{d}{=}\nu$ are  independent  
random variables with a common distribution
 \begin{equation}\label{fs}
f(s)=Es^\nu=\sum_{k\ge0} p_k s^k.
\end{equation}
In terms of probability generating functions,  the branching property \eqref{def:model} yields
\begin{equation}\label{iter}
Es^{Z(t)}=f(t,s),\quad  f(0,s)=s,\quad f(t+1,s)=f(f(t,s)),\quad t\ge0.
\end{equation}
There are two types of trajectories for this simple demographic  model: a GW-process either becomes extinct at time $T_0=\inf\{t\ge1: Z(t)=0\}$ or $Z(t)\to\infty$ as $t\to\infty$. It is well known that the corresponding probability of extinction $q=P(T_0<\infty)$ is given by the smallest non-negative root of the equation $f(s)=s$, see  \cite[Ch I.5]{Athreya-Ney}. Much of the theory of branching processes is devoted to the limit behavior of $Z(t)$ conditioned on $T_0>t$ as $t\to\infty$, see  \cite{Melard-2012}.

This paper deals with {\it defective GW-processes} having  $f(1)\in(0,1)$. We treat the defect $\eps=1-f(1)$ of the reproduction law \eqref{fs} as the  probability that a given particle at a given time $t$ sends the Markov chain $Z(t+1)$ to an  additional  graveyard state $\Delta$.
Thus, a defective GW-process becomes a Markov chain with a countable  state space $\mathbb N_\Delta= \{0,1,\ldots\}\cup\{\Delta\}$. Two of the states are absorbing: the process either becomes extinct at time $T_0$, or is stopped at time $T_\Delta=\inf\{t\ge1: Z(t)=\Delta\}$. If $T=T_0\wedge T_\Delta$ denotes the ultimate absorption time, then for some $q\in [0,1)$,
\[P(T_0<\infty)=q,\quad P(T_\Delta<\infty)=1-q,\quad P(T<\infty)=1.\]
Applying the graveyard absorption properties 
$$\Delta+x=\Delta, \quad x\in\mathbb N_\Delta, \qquad s^\Delta=0,\quad s\ge0,\qquad \sum_{j=1}^\Delta x_j=\Delta, \quad x_j\in\mathbb N_\Delta,$$
to the recursion \eqref{def:model}, we obtain again \eqref{iter} implying $f(q)=q$.
Clearly,
 $P(Z(t)=\Delta)=1-f(t,1)$, and  if $q=0$, then $T=T_\Delta$.
It is straightforward to see that
\[E(s^{Z(t)};T_\Delta>t)=f(t,s),\quad E(s^{Z(t)};T>t)=E(s^{Z(t)};T_0>t)=f(t,s)-f(t,0),\]
since
\[E(s^{Z(t)};T\le t)=E(s^{Z(t)};T_0\le t)=P(Z(t)=0)=f(t,0).\]
This implies,
\begin{align*}
 P(t<T_\Delta<\infty)&=f(t,1)-q,\\
P(t<T_0<\infty)&=q-f(t,0), \\
P(T>t)&=f(t,1)-f(t,0).
\end{align*}

The main focus of this paper is the asymptotic distribution of  $Z(t-k)$ conditioned on the survival event $\{T>t\}$ as $t\to\infty$, with $k\in[0,t]$ either being fixed or going to infinity.
In Section \ref{secA} we present limit theorems assuming that the reproduction law $f(\cdot)$ is fixed. We will see that with fixed $f(\cdot)$, there are two different asymptotic regimes depending on whether $\gamma>0$ or $\gamma=0$, where  $\gamma=f'(q)$. The proofs of the results of Section \ref{secA} are collected in Section \ref{pr1}.

In realistic settings, the defect $\eps$ of the reproduction is small and therefore it is interesting to find asymptotic results as $t\to\infty$ and $\eps\to0$. To address this issue in Sections \ref{Sext} and \ref{theta-branching} we consider sequences of defective GW-processes $(Z_n(\cdot))_{n\ge1}$ governed by reproduction laws $f_n(\cdot)$ such that $\eps_n\to0$ as $n\to\infty$ and  $ f_n(s)\to \hat f(s)$ uniformly over $s\in[0,1]$, where $\hat f(1)=1$. It turns out that with this approach, a key parameter determining the limit behavior is not $\gamma$ as in Section \ref{secA}, but rather $\hat m=\hat f'(1)$.  We assume $\hat m>1$ and even consider the case $\hat m=\infty$.
The proofs of the results of Sections \ref{Sext} and \ref{theta-branching}  are collected in Section \ref{pr2}.

Earlier, a special subclass of the defective GW-processes, the  so-called GW-processes with killing, was studied in \cite{Karlin-Tavare-1982, Pakes-1984}. A GW-process with killing
has a reproduction law of the form $f(s)=g(\alpha s)$, where  $g(\cdot)$ is a non-defective generating function and $\alpha\in(0,1)$. In this case $f(1)\in(0,1)$ and $f(s_0)=1$ for $s_0=1/\alpha>1$.
To see a counterexample  violating the latter restriction, consider
\begin{equation}\label{xet}
 f_0(s)=1-(p_1\sqrt{1-s}+1-p_1)^2,\quad s\in[0,1],
\end{equation}
having $f_0(1)=p_1(2-p_1)$ and
\[f_0(t,s)=1-(p_1^t\sqrt{1-s}+1-p_1^t)^2.\]
Since $f'_0(1)=\infty$, the generating function $f_0(s)$ is not defined for $s>1$.
Example \eqref{xet} belongs to a parametric family of defective generating functions with  explicit  iterations: in \cite{Sagitov-Lindo-2016} the corresponding family of GW-processes is called theta-branching processes.
We turn to the theta-branching processes in Section \ref{theta-branching}. A broad class of continuous time defective branching processes was investigated in  \cite{Sagitov-2017}.

Defective GW-processes arise naturally in the framework of some special non-defective GW-processes with countably many types. For example, the authors of \cite{Braunsteins-Hautphenne-2017} construct an embedded defective GW process in which absorption in the graveyard state corresponds to local survival of the GW-process with countably many types, and absorption in state 0 corresponds to its global extinction. In another multi-type setting, \cite{Sagitov-Serra-2009} treat the defect $\eps$ as the probability of a favorable mutation allowing a population of viruses to escape extinction.

Notice that the defective GW-processes can be put into the framework of $\phi$-branching processes  using a random control function
\[
\phi(k)=\left\{
\begin{array}{llr}
  k&  \text{with probability} & (1-\eps)^k,  \\
 \Delta  & \text{with probability} &1- (1-\eps)^k,
\end{array}
\right. \quad k\ge0,
\]
cf. \cite{Yanev-75}.  Indeed, in the defective case, the branching property \eqref{def:model} can be rewritten as
\[
Z(t+1)=\sum_{j=1}^{\phi_t(Z(t))}\tilde \nu_{t,j},\quad t=0,1,\ldots,
\]
where  $\phi_t(\cdot)\stackrel{d}{=}\phi(\cdot)$.
Here the common distribution of the random variables $\tilde \nu_{t,j}$ has a proper probability generating function $f(\cdot)/f(1)$. For a given small value of $\eps$, the control function gets a chance  to stop the growth of a non-defective GW-process, when the population size $k$ becomes inverse-proportional to $\eps$, that is when the stopping probability $1- (1-\eps)^k$ is approximated by $1- e^{-\eps k}$.


\section{Limit theorems with fixed reproduction law}\label{secA}

In this section we assume that the defective reproduction law $f(\cdot)$ is fixed while the observation time $t$ tends to infinity. Recall that $q\in[0,1)$ is defined by $q=f(q)$ and $\gamma=f'(q)$. Observe that $\gamma\in[0,1)$ and denote
$$l=\min\{k\ge0:p_k> 0\}.$$
Clearly, $q=0$ if and only if ${l}\ge1$, and $\gamma=0$ if and only if ${l}\ge2$.  Define $\pi_t=\gamma^{t}$ for $l=0,1$, and
$$\pi_t=\prod_{k=0}^{t-1}p_{l}^{\ l^k}=p_{l}^{a_t},\quad a_t={l^{t}-1\over l-1},\quad t\ge1,$$
for $l\ge2$. Observe that given $l\ge1$,  the minimal $t$-th generation size is $l^t$ and
$$P(Z(t)=l^t)=\pi_t,$$

\begin{prop}\label{HR}
Consider iterations $f(t,\cdot)$ of a defective probability generating function $f(\cdot)$.
\begin{enumerate}[label=(\alph*)]
\item If $\gamma>0$, then for each $s\in [0,1]$,
$$f(t,s)-q\sim (s-q)H(s)\pi_t,\quad t\to\infty,$$
where  $H(\cdot)$ is a generating function defined as
$$H(s)=\prod_{j=0}^\infty h(f(j,s)),\quad h(s)={f(s)-q\over (s-q)\gamma},$$
and having $H(q)=1$, $H(1)<\infty$.
\item If $\gamma=0$, then  for each $s\in [0,1]$,
 \begin{equation*}
 f(t,s)\sim(sR(s))^{{l}^t} \pi_t , \quad t\to\infty,
\end{equation*}
where $R(\cdot)$ is a generating function defined as
$$R(s)=\prod_{j=0}^\infty (b(f(j,s)))^{ l^{-j-1}},\quad b(s)={f(s)\over p_ls^l},$$
and having
$$1=R(0)<R(1)<p_l^{-1/(l-1)}.$$
\end{enumerate}
\end{prop}

Proposition \ref{HR} indicates that there are two different asymptotic regimes depending on whether $\gamma>0$ or $\gamma=0$. An immediate consequence of Proposition \ref{HR}-a  is
\begin{align}\label{tail1}
 \gamma^{-t} P(T>t)&\to
 qH(0)+(1-q)H(1), \quad t\to\infty,
 \end{align}
which implies
\begin{equation*}
P(T=t+k | T\geq t)\to (1-\gamma)\gamma^k, \quad k\ge1.
\end{equation*}
As it is shown next by Theorem \ref{MC}, devoted to the case $\gamma>0$, relation  
\begin{equation}\label{qj}
 \frac{(s-q)H(s)+qH(0)}{(1-q)H(1)+qH(0)}=\sum_{j\ge1} q_js^j
\end{equation}
defines an important proper distribution $(q_j)_{j\ge1}$. 

\begin{thm}\label{MC}
Consider a defective GW-process with $\gamma> 0$.
\begin{enumerate}[label=(\alph*)]
\item The asymptotic relation \eqref{tail1} holds, and for  $k\ge0,  j\ge1$,
\begin{equation*}
P(Z(t-k)=j|T>t)\to q_{k,j} ,\quad t\to\infty,
\end{equation*}
where $(q_{k,j})_{j\ge1}$ is a proper probability distribution defined by
\begin{equation}\label{qkqj}
q_{k,j}=q_{j}\gamma^{-k}(f_{k}(1)^{j}-f_{k}(0)^{j}),
\end{equation}
so that  $q_{0,j}\equiv q_j$ are given by \eqref{qj}.
\item For $j_0\ge1,\ldots,j_k\ge1$,  $k\ge0$,
\begin{equation*}
P(Z(t)=j_0,\ldots,Z(t-k)=j_k| T> t) \to q_{k,j_k}Q_{j_k,j_{k-1}}^{(k)}
\cdots Q_{j_{1},j_{0}}^{(1)},\quad \quad t\to\infty,
\end{equation*}
where
\begin{equation*}
Q_{ij}^{(k)}= \frac{f_{k-1}(1)^{j}-f_{k-1}(0)^{j}}{f_{k}(1)^{i}-f_{k}(0)^{i}}P_{ij},\quad \sum_{j\ge1}Q_{ij}^{(k)}=1, \quad i\ge1,
\end{equation*}
is a transformation of the time-homogeneous  transition probabilities
$$P_{ij}=P(Z(t+1)=j|Z(t)=i).$$
\end{enumerate}
\end{thm}
We see that in the case $\gamma>0$, the conditional branching process asymptotically behaves as a time-inhomogeneous Markov chain.
 Observe that given $q\in(0,1)$, the limit towards the past
\begin{equation*}
Q_{ij}^{(k)}\to\frac{P_{ij}j q^{j-i}}{\gamma i},\qquad k\to\infty,
\end{equation*}
recovers the well known formula for the so-called Q-process, see  \cite[Ch I.14]{Athreya-Ney} and \cite{Sagitov-Lindo-2016}.


\begin{figure}
\centering
  \epsfig{file=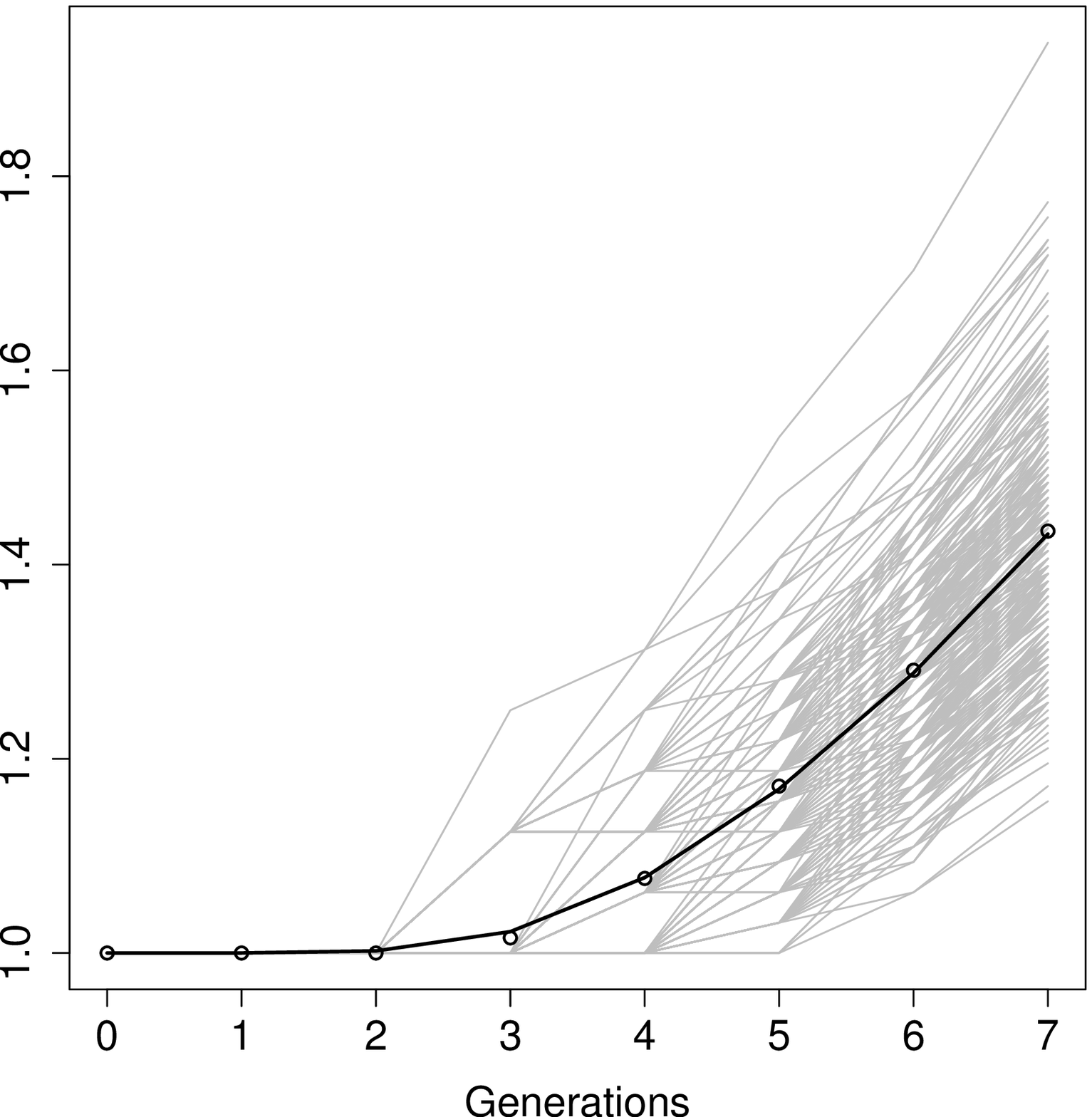, width=0.45\linewidth}\hspace{0.05\linewidth}
  \epsfig{file=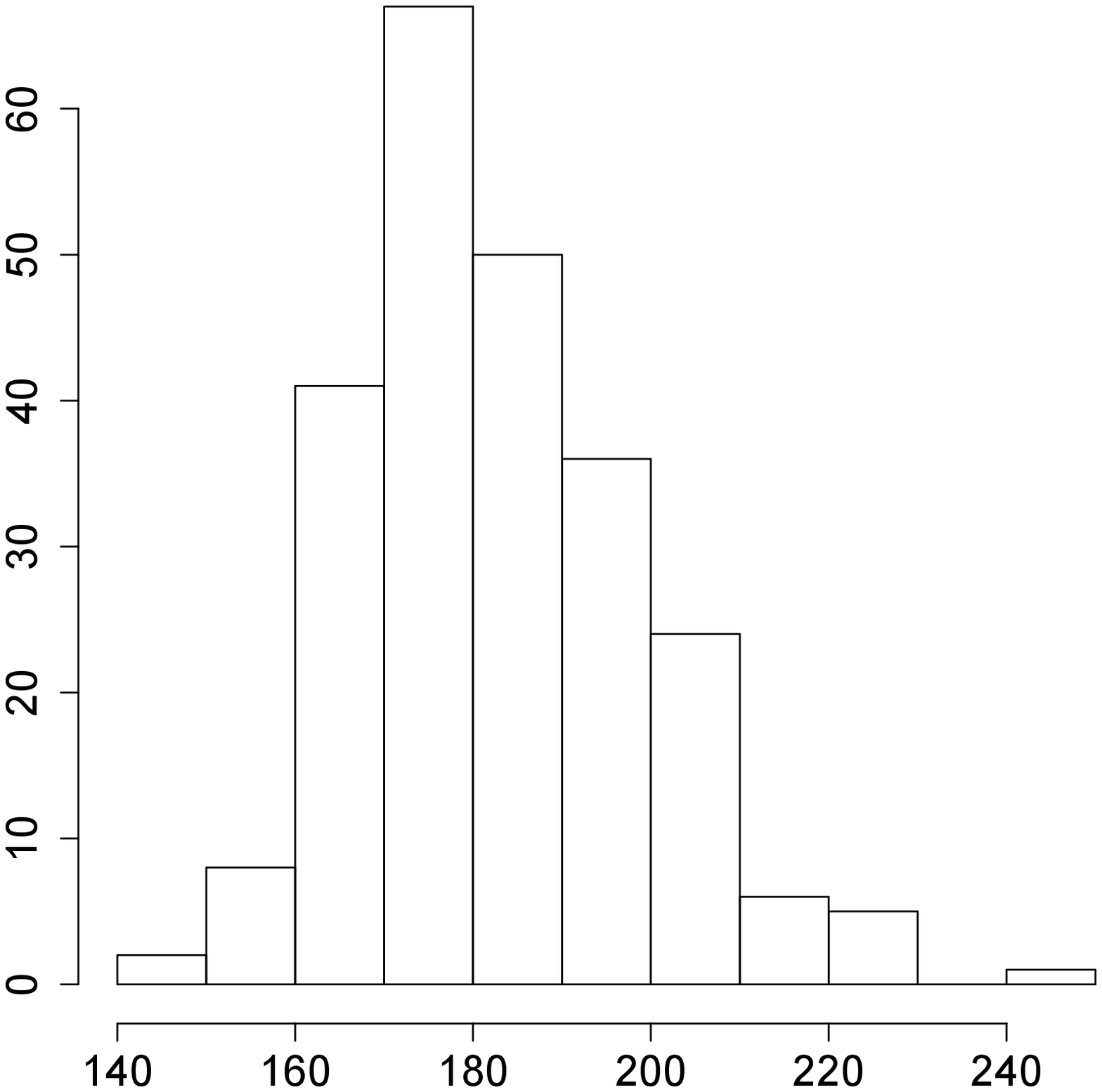, width=0.45\linewidth}
  \captionof{figure}{Simulation results for $f(s)=0.7 x^2+ 0.2 x^3$ and $t=7$.
 {\it Left panel.} Grey lines represent the vectors $(Z(0), 2^{-1}Z(1),\ldots , 2^{-t}Z(t))$ for 240 successful simulations having $T>t$. The thick black line shows the limit vector $(c(t),c(t-1),\ldots,c(0))$ suggested by Theorem \ref{LD}, which provides with a good approximation for the average trajectory (shown by circles)  even for the small observation time $t=7$.
 {\it  Right panel.} The histogram presents the observed values $Z(t)$ in the successful simulations. 
 }\label{f1}
\end{figure}

On the other hand, for $\gamma=0$, Proposition \ref{HR}-b gives a much faster decay of the tail distribution
\begin{align}\label{tail2}
 P(T>t)&
  \sim\pi_tR(1)^{l^t}=p_l^{-\frac{1}{l-1}} \rho^{l^t},\quad t\to\infty,
\end{align}
where $\rho=p_l^{\frac{1}{l-1}}R(1)\in(0,1)$. This yields
$P(T=t | T\geq t)\to 1$.
The next Theorem \ref{LD} establishes a conditional weak law of large numbers for ${l}^{t-k} Z(t-k)$. 
\begin{thm}\label{LD}
 Consider a defective GW-process with  $\gamma=0$. Then the asymptotic relation \eqref{tail2} holds and for the normalized process  $Y(t)={l}^{-t}Z(t)$, we have the following results concerning its expectation and variance.
\begin{enumerate}[label=(\alph*)]
\item If $f'(1)<\infty$, then uniformly over $0\le k\le t$,
\begin{align*}
E(Y(k)|T>t)-c(t-k)&\to 0,\quad t\to\infty,
\end{align*}
where in terms of $\bar R(s)=R'(s)/R(s)$,
\begin{equation}\label{ck}
 c(k)=1+f(k,1)\bar R(f(k,1)), \quad k=0,1,\ldots,
\end{equation}
is a strictly decreasing sequence with
\[ 1<\ldots<c(k+1)< c(k)<c(k-1)<\ldots<c(1)<c(0)<\infty.
\]
\item If $f''(1)<\infty$, then uniformly over $0\le k\le t$,
$$Var(Y(k)|T>t)\to0,\quad t\to\infty.$$
\end{enumerate}
\end{thm}

According to Theorem \ref{LD}-b, if $f''(1)<\infty$, then conditionally on $T>t$, we have convergence in probability $Y(t-k)\to c(k)$ as $k\ge0$ is fixed and $t\to\infty$, and convergence in probability $Y(k)\to 1$ as $t-k\to\infty$. This indicates that
being conditioned on survival, the reproduction regime prefers the minimal offspring number ${l}$, especially at early times (see Figure  \ref{f1}).


\section{Extendable defective GW-processes}\label{Sext}

Suppose $f(r)=r$ for some $r>1$, so that necessarily $f(1)<1$ (see Figure  \ref{f2}). In this case the corresponding defective GW-process $Z(\cdot)$ could be called an extendable GW process because the usual range $0\le s\le 1$ for the reproduction generating function $f(s)$ can be extended to $0\le s\le r$. The transformed function
\[\hat f(s)=r^{-1}f(rs),\quad s\in[0,1],\quad \hat f(1)=1,\]
generates a proper reproduction distribution $\hat p_k=r^{k-1}p_k$ with mean $\hat m=\hat f'(1)=f'(r)$. Denote by $\hat Z(\cdot)$  the GW-process with the reproduction law $\hat f(\cdot)$.
If  $\hat m\in(1,\infty)$, then by Theorem 3 in \cite[Ch I.10]{Athreya-Ney}, there exists a sequence $C(t)\to\infty$, $t\to\infty$ such that $\hat Z(t)/C(t)\to W$ a.s., where $P(W>0)=1-\hat q$ and $\hat q=q/r$. In this case,  for any given $\lambda\ge0$, we have a positive finite limit
\begin{equation}\label{psi}
E(e^{-\lambda \hat Z_n(t)/C(t)}|\hat T_0>t)\to \Psi(\lambda),\quad t\to\infty,
\end{equation}
where $\Psi(\lambda)=E(e^{-\lambda W}|W>0)$.
On the other hand, if $\hat m=\infty$, then by \cite{Darling},
\begin{equation}\label{phi}
P(b^{-t}\ln \hat Z(t)\le u|\hat T_0>t)\to \psi(u),\quad u\in(0,\infty),
\end{equation}
provided the following condition holds
\[g'(x)=ax^{b-1}(1+O(x^\delta)),\quad x\to0,\quad a>0,\quad b>1,\quad \delta>0.\]
Here  $g(\cdot)=G_{-1}(\cdot)$ is the  inverse function of $G(x)=1-f(1-x)$,
and the limit $\psi(\cdot)$ in \eqref{phi} is continuous and strictly monotonic increasing function such that
$$\psi(u)\to0,\quad u\to0+,\qquad \psi(u)\to1,\quad u\to\infty.$$

\begin{figure}
\centering
 \epsfig{file=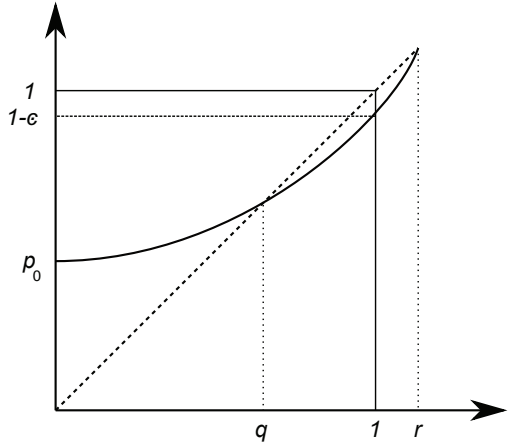, width=0.45\linewidth}
 \caption{Extendable generating function $f(\cdot)$.}\label{f2}
\end{figure}

\begin{thm}\label{extGW}
Let $\hat f(\cdot)$ be a probability generating function for a proper  reproduction law.
Consider a sequence of defective GW-processes  $Z_n(\cdot)$ corresponding to the sequence of  reproduction laws
\begin{equation}\label{An}
 f_n(s)=r_n\hat f(s/r_n),\quad r_n>1,\quad n\ge1.
\end{equation}
\begin{enumerate}[label=(\alph*)]
\item Suppose $\hat m\in(1,\infty)$ so that \eqref{psi} holds. If for some sequence $t_n\to\infty,$
\[(r_n-1)C(t_n)\to x\in(0,\infty),\]
then
\[P\left(T_n>t_n\right)\to (1-\hat q)\Psi\left(x\right),\]
and for each $\lambda\ge0$,
\begin{equation}\label{coco}
 E(e^{-\lambda Z_n(t_n)/C(t_n)}|T_n>t_n)\to \Psi(\lambda+x)/\Psi(x),\quad n\to\infty.
\end{equation}
\item Suppose $\hat m=\infty$ and \eqref{phi} holds. If for some sequence $t_n\to\infty,$
\[b^{-t_n}\ln (r_n-1)^{-1}\to y,\quad y\in(0,\infty),\quad n\to\infty,\]
then 
\[P\left(T_n>t_n\right)\to (1-\hat q)\psi(y),\]
and for $u\in[0,y]$,
\[P(b^{-t_n}\ln Z_n(t_n)\le u| T_n>t_n)\to \psi(u)/\psi(y),\quad n\to\infty.\]
\end{enumerate}
\end{thm}
Theorem \ref{extGW}-a should be compared to  \cite[Theorem 3.4]{Pakes-1984} concerning a sequence of GW-processes with killing:   if $Z_n(\cdot)$ has a reproduction law of the form $f_n(s)=\hat f(\alpha_n s)$, where 
$\hat f(1)=1$, $\hat f'(1)\in(1,\infty)$,  and
$$ (1-\alpha_n)C(t_n)\to (\hat m-1)x/\hat m,\quad n\to\infty,$$
then the same weak convergence result \eqref{coco} holds. The proof of Theorem \ref{extGW} given in Section \ref{pr2} is more straightforward than the proof of  \cite[Theorem 3.4]{Pakes-1984}, which demonstrates  the advantage of dealing with the extendable GW-processes.

\section{Explicit limits for defective theta-branching processes}\label{theta-branching} 

The main assumption of Section \ref{Sext}  is quite restrictive on the mode of convergence $f_n(\cdot)\to\hat f(\cdot)$, namely, condition \eqref{An} requires that the sequence $f_n(\cdot)$ has a common shape of the reproduction laws and only a scale parameter $r_n\to1$ is changing as $n\to\infty$. In this section we take a step towards a more general setting for the convergence $f_n(\cdot)\to\hat f(\cdot)$. We focus on the parametric family of the theta-branching processes introduced in \cite{Sagitov-Lindo-2016}. Our Propositions \ref{peps}, \ref{peps2} and \ref{peps3} give  explicit expressions for the corresponding limit distributions.

Proposition \ref{peps} is a counterpart of Theorem \ref{extGW}-a in terms of a sequence of  extendable GW-processes whose generating functions are explicitly characterized by four parameters
$$(\theta_n, q_n,\gamma_n,r_n)\in(0,1]\times[0,1)\times(0,1)\times(1,\infty)$$
as follows
\begin{align*}
f_n(t,s)=
r_n-\left[\gamma_n^t (r_n-s)^{-\theta_n}+(1-\gamma_n^t)(r_n-q_n)^{-\theta_n}\right]^{-1/\theta_n}, \quad  s\in [0,r_n],
\end{align*}
In agreement with our previous notation, $q_n$ is the extinction probability and  $\gamma_n=f_n'(q_n)$. These are defective GW-processes with the defect value
\[\eps_n=\left[\gamma_n(r_n-1)^{-\theta_n}+(1-\gamma_n)(r_n-q_n)^{-\theta_n}\right]^{-1/\theta_n}-(r_n-1).\]

\begin{prop}  \label{peps}
Fix a triplet $(\theta,q,\gamma)\in(0,1]\times[0,1)\times(0,1)$ and consider an above described  sequence of defective theta-branching processes $Z_n(\cdot)$ with
$$(\theta_n,\gamma_n,q_n,r_n)\to(\theta,\gamma,q,1),\quad n\to\infty.$$
Denote $m_n=f'_n(1)=\gamma_n^{-1/\theta_n}$ and assume that for some $t_n\to\infty$,
\begin{equation}\label{aseps}
 (r_n-1)m_n^{t_n}\to x\in(0,\infty),\quad n\to\infty.
\end{equation}
\begin{enumerate}[label=(\alph*)]
\item As $n\to\infty$,
\[P\left(T_n>t_n\right)\to (1-q)\Psi\left(x\right),\]
where
\begin{equation}\label{psih}
 \Psi(\lambda)=1-\left[1+(1-q)^{\theta}\lambda ^{-\theta}\right]^{-1/\theta},\quad \lambda\ge0.
\end{equation}
\item If $k\ge0$ and $t_n-k\to\infty$,  then  for each $\lambda\ge0$,
\begin{align*}
 E\left(\exp\{-\lambda m_n^{k-t_n} Z_n(t_n-k)\}|T_n>t_n\right)\to {\Psi\left(x+\lambda\right)\over \Psi\left(x\right)},\quad n\to\infty.
\end{align*}
\end{enumerate}
\end{prop}

Under the conditions of Proposition  \ref{peps} we have $f_n(s)\to\hat f(s)$, where
\begin{align}\label{gs}
\hat f(s)=1-\left[\gamma (1-s)^{-\theta}+(1-\gamma)(1-q)^{-\theta}\right]^{-1/\theta}.
\end{align}
For the corresponding supercritical GW-process having the offspring mean $\hat m=\gamma^{-1/\theta}$,
it is straightforward to check that the limit Laplace transform
\[E (e^{-\lambda \hat Z(t)\hat m^{-t}}|\hat T_0>t)=1-{1-\hat f(t,e^{-\lambda \gamma^{t/\theta}})\over 1-\hat f(t,0)}\to\Psi(\lambda),\quad t\to\infty,\]
is given by \eqref{psih}.
Since
\[\eps_n\sim(\gamma^{-1/\theta}-1)(r_n-1),\quad n\to\infty,\]
the first part of Proposition \ref{peps} essentially says that for a given small $\eps$, the absorption time $T$ of a defective theta-branching process with $\theta\in(0,1]$ is of order $\theta\log_\gamma \eps$. Observe that the new normalization $m_n^{t_n}$ may not be asymptotically equivalent to the normalization $\hat m^{t_n}$ suggested by  Theorem \ref{extGW}-a under an additional "xlogx" condition.

The next two propositions deal with two different sequences $f_n(\cdot)$ converging to the same limit reproduction law given by
\begin{align}\label{stem}
\hat f(s)=1-(1-q)^{1-\gamma}(1-s)^{\gamma}, \quad  s\in [0,1],
\end{align}
with $q\in[0,1)$, $\gamma\in(0,1)$, $\hat f(1)=1$, and $\hat m=\hat f'(1)=\infty$. Plugging $s=\exp\{-\lambda e^{-u\gamma^{-t}}\}$ into
\[\hat f(t,s)=1-(1-q)^{1-\gamma^t}(1-s)^{\gamma^t},\]
it is straightforward to find a convergence
\begin{align*}
 P\left(\gamma^{t}\ln \hat Z(t)\le u|\hat T_0>t\right)\to 1-e^{-u},\quad u\ge 0
\end{align*}
to a standard exponential distribution.

\begin{prop}  \label{peps2}
Consider a sequence of defective GW-processes $Z_n(\cdot)$ having the following reproduction laws
\begin{align*}
f_n(s)=
r_n-(r_n-q_n)^{1-\gamma_n}(r_n-s)^{\gamma_n},\quad s\in[0,r_n),
\end{align*}
with $(q_n,\gamma_n,r_n)\in[0,1)\times(0,1)\times(1,\infty)$. Suppose that for some $(q,\gamma)\in[0,1)\times(0,1)$,
$$(q_n,\gamma_n,r_n)\to(q,\gamma,1)\quad n\to\infty,$$
and that for some $t_n\to\infty$,
\begin{equation}\label{lug}
 \gamma_n^{t_n}\ln(r_n-1)^{-1}\to  y\in(0,\infty),\quad n\to\infty.
\end{equation}
\begin{enumerate}[label=(\alph*)]
\item As $n\to\infty$,
\begin{equation}\label{lum}
P\left(T_n>t_n\right)\to (1-q)(1-e^{-y}).
\end{equation}
\item If $k\ge0$ and $t_n-k\to\infty$, then
\begin{align}\label{same}
 P\left(\gamma_n^{t_n-k}\ln Z_n(t_n-k)\le u|T_n>t_n\right)\to {1-e^{-u}\over 1-e^{-y}},\quad 0\le u\le y.
\end{align}
\end{enumerate}
\end{prop}

Since in this parametric case the defect size has the asymptotic value
\[\eps_n\sim(1-q)^{1-\gamma}(r_n-1)^\gamma,\quad n\to\infty,\]
the first part of Proposition \ref{peps2} essentially says that for a given small defect value $\eps$, the absorption time of a defective theta-branching process with $\theta\in(0,1]$ is of order $\ln\ln \eps^{-1}$.
\begin{prop}  \label{peps3}
Consider a sequence of defective GW-processes $Z_n(t)$ having the following reproduction laws
\begin{align*}
f_n(s)=
A_n-\left[\gamma_n (A_n-s)^{|\theta_n|}+(1-\gamma_n)(A_n-q_n)^{|\theta_n|}\right]^{1/|\theta_n|}, \quad s\in [0,A_n],
\end{align*}
where $( \theta_n,q_n,\gamma_n,A_n)\in(-1,0)\times[0,1)\times(0,1)\times [1,\infty)$. Suppose that for some $(\gamma,q)\in(0,1)\times[0,1)$,
$$(\theta_n,\gamma_n,q_n,A_n)\to(0,\gamma,q,1), \quad n\to\infty,$$
in such a way that  for some $t_n\to\infty$,
\begin{align}
|\theta_n|\ln (A_n-1)^{-1}&\to a\in (0,\infty],\label{eq:cond-peps3}\\
 \gamma_n^{t_n}|\theta_n|^{-1}&\to  y\in(0,\infty),\quad n\to\infty.\label{luga}
\end{align}
\begin{enumerate}[label=(\alph*)]
\item As $n\to\infty$, 
\begin{equation*}
P\left(T_n>t_n\right)\to (1-q)(1-e^{-y(1-e^{-a})}).
\end{equation*}
\item [(b$_1$)] If $k\ge0$ is fixed , then putting $\hat u(x)={-x}\ln(1-u/x)$,
\begin{equation*}
P\left(\gamma_n^{t_n-k}\ln Z_n(t_n-k)\le \hat u(y\gamma^{-k})|T_n>t_n\right)\to {1-e^{-u}\over 1-e^{-y(1-e^{-a})}},\quad 0\le u < y(1-e^{-a}).
\end{equation*}
\item [(b$_2$)] If $k\to\infty$,  $t_n-k\to\infty$, then
\begin{align*}
 P\left(\gamma_n^{t_n-k}\ln Z_n(t_n-k)\le u|T_n>t_n\right)\to {1-e^{-u}\over 1-e^{-y(1-e^{-a})}},\quad 0\le u < y(1-e^{-a}).
\end{align*}
\end{enumerate}
\end{prop}

Here, $\eps_n\sim(1-q)(1-\gamma)^{1/|\theta_n|}$ and by Proposition \ref{peps3}-a,  given a small defect value $\eps$, the absorption time is again of order $\ln\ln \eps^{-1}$. If $A_n\equiv1$, then $a=\infty$, and convergence in Proposition \ref{peps3}-a is given by  \eqref{same}.
To see a connection of the convergence in Proposition \ref{peps3}-b$_1$  to that of Proposition \ref{peps3}-b$_2$, notice that $\hat u(x)\to u$, as $x\to\infty$.


\section{Proofs of Proposition \ref{HR} and Theorems \ref{MC} and \ref{LD}}\label{pr1}

\subsection{Proof of Proposition \ref{HR}}

Assume $\gamma>0$. Putting
$$H_t(s)={f(t,s)-q\over(s-q) \gamma^t },\quad 0\le s\le 1,\quad t\ge1,$$
observe that
$$H_t(s)=\prod_{j=0}^{t-1}h(f(j,s)),\quad h(s)={f(s)-q\over (s-q)\gamma}.$$
It is easy to check that  $h(\cdot)$ is a generating function with $h(q)=1$. (In fact, ${f(s)-f(q)\over s-q}$ is a tail generating function naturally linked to the reproduction law $f(\cdot)$, see \cite{Sagitov-2017}.) It follows that $H_t(\cdot)$
is also a generating function such that $H_t(q)=1$.

Since $h(f(t,s))<1$  for $s<q$, and $h(f(t,s))>1$  for $s>q$, we conclude that  $H_{t+1}(s)<H_t(s)$ for $s<q$, and $H_{t+1}(s)>H_t(s)$ for $s>q$.
Due to this monotonicity property, we have $H_t(s)\to H(s)$, as $t\to\infty$, where the limit function $H(s)$ has the stated form.

To finish the proof of Proposition  \ref{HR}-a it remains to show that $H(1)<\infty$ or equivalently,
\[\sum_{j=1}^\infty (h(f(j,1))-1)<\infty.\]
The last is indeed true because
$$h(f(t,1))-1\le \left(1-{\eps\over 1-q}\right)^t c,\quad t>t_0,$$
for some finite $c$ and $t_0$.
This upper bound is justified using two observations: on one hand, we have
\[{h(s)-h(q)\over s-q}\to {f''(q)\over \gamma}\in(0,\infty),\quad s\to q,\]
and on the other hand,
\[f(t,1)-q\le (1-q)\big(1-{\eps\over 1-q}\big)^t,\]
which is due to the following  convexity property of $f(\cdot)$
\[f(s)\le q+ (s-q){ 1-q-\eps\over 1-q},\quad s\in[q,1].\]

Assume now $\gamma=0$, or equivalently $l\ge2$. By iterating the function $f(s)=p_ls^l b(s)$, we get the following representation
\begin{equation}\label{Rt}
f(t,s)=\pi_t(s R_t(s))^{l^t},\quad R_t(s)=\prod_{j=1}^{t} \big(b(f(j-1,s))\big)^{l^{-j}},\quad t\ge0.
\end{equation}
A straightforward adjustment to the defective case $f(1)<1$ of the argument used in \cite[Prop. 3]{athreya1994} shows that the sequence of monotonely increasing functions $R_t(\cdot)$
has a well defined limit
$$R(s)=\lim_{t\to\infty} R_t(s)=\prod_{j=1}^{\infty} b(f(j-1,s))^{l^{-j}},\quad s\in [0,1],$$
and moreover, that
$$\lim_{t\to\infty} \left(R_t(s)/R(s)\right)^{l^t}=1.$$
This proves the main assertion of Proposition  \ref{HR}-b. It remains to verify the stated upper bound for $R(1)$ which in terms of $\rho=p_l^{\frac{1}{l-1}}R(1)$, is equivalent to the inequality $\rho< 1$.
Since $f(t,1)\to q=0$, the relation
\[f(t,1)\sim\pi_tR(1)^{l^t}=p_l^{-\frac{1}{l-1}} \rho^{l^t},\quad t\to\infty\]
indeed implies that $\rho< 1$. This also gives \eqref{tail2}.

\subsection{Proof of Theorem \ref{MC}}
We will need the following relations
\begin{align}
 P(T>t|Z(k)=i)&=f^i(t-k,1)-f^i(t-k,0),\label{Tank1}\\
E(s^{Z(k)}| T>t)&=\frac{f(k,s f(t-k,1)))-f(k,s f(t-k,0)))}{f(t,1)-f(t,0)},\label{Tank2}
\end{align}
holding for $0\le k\le t<\infty$,  $s\in [0,1]$. Relation \eqref{Tank1} follows from
$$\{T>t\}=\{T_\Delta>t\}\setminus\{T_0\le t\}$$
and
\begin{align*}
 P(T_\Delta>t|Z(k)=i)&=P(Z(t-k)\neq \Delta)^i=f^i(t-k,1),\\
 P(T_0\le t|Z(k)=i)&=P(Z(t-k)=0)^i=f^i(t-k,0).
\end{align*}
Relation \eqref{Tank2} is obtained using \eqref{Tank1} as follows

\begin{align*}
E(s^{Z(t)}| T>t+k)
&=\frac{E(s^{Z(t)} P(T>t+k|Z(t)))}{P(T>t+k)}\\
&=\frac{E((s f(k,1))^{Z(t)})-E[(s f(k,0))^{Z(t)})}{f(t+k,1)-f(t+k,0)}\\
&=\frac{f(t,s f(k,1)))-f(t,s f(k,0)))}{f(t+k,1)-f(t+k,0)}.
\end{align*}

Applying  \eqref{Tank2} and Proposition \ref{HR}-a, we get
\begin{align*}
E(s^{Z(t-k)}| T>t)&=\frac{f(t-k,sf(k,1))-f(t-k,sf(k,0))}{f(t-k,f(k,1))-f(t-k,f(k,0))}\\
&\to\frac{(sf(k,1)-q)H(sf(k,1))-(sf(k,0)-q)H(sf(k,0))}{(f(k,1)-q)H(f(k,1))-(f(k,0)-q)H(f(k,0))}.
\end{align*}
In particular,
\begin{equation*}
E(s^{Z(t)}| T>t)\to\frac{(s-q)H(s)+qH(0)}{(1-q)H(1)+qH(0)}=\sum_{j=1}^\infty q_js^j.
\end{equation*}
Thus, $P(Z(t-k)=j|T>t)\to q_{k,j}$ with
\begin{align*}
\sum_{j=1}^\infty q_{k,j}s^j&=\frac{(sf(k,1)-q)H(sf(k,1))-(sf(k,0)-q)H(sf(k,0))}{(f(k,1)-q)H(f(k,1))-(f(k,0)-q)H(f(k,0))}.
\end{align*}
Modifying the denominator by a repeated use of the relation
\[(f(s)-q)H(f(s))=\gamma(s-q)H(s),\]
we find
\begin{align*}
\sum_{j=1}^\infty q_{k,j}s^j
&=\gamma^{-k}\frac{(sf(k,1)-q)H(sf(k,1))-(sf(k,0)-q)H(sf(k,0))}{(1-q)H(1)+q H(0)}\\
&=\gamma^{-k}\left(\sum_{j=1}^\infty q_{j}(sf(k,1))^j-\sum_{j=1}^\infty q_{j}(sf(k,0))^j\right),
\end{align*}
which implies \eqref{qkqj} thereby finishing the proof of Theorem \ref{MC}-a.

Turning to the proof of Theorem \ref{MC}-b, observe that by \eqref{Tank1},
\begin{align*}
P(T> t|Z(t)=j_0,\ldots,Z(t-k)=j_k)
&=f(1)^{j_0}-f(0)^{j_0},
\end{align*}
implying
\begin{align*}
P(Z(t)=j_0,\ldots,Z(t-k)=j_k; T> t)
&=P(Z(t)=j_0,\ldots,Z(t-k)=j_k)(f(1)^{j_0}-f(0)^{j_0}).
\end{align*}
Similarly,
\begin{align*}
P(Z(t-k)=j_k; T> t)
&=P(Z(t-k)=j_k)(f(k,1)^{j_k}-f(k,0)^{j_k}),
\end{align*}
which gives
\begin{align*}
P(Z(t-k)=j_k)
&\sim  q_{k,j_k}(f(k,1)^{j_k}-f(k,0)^{j_k})^{-1}P(T> t).
\end{align*}
Therefore,  by the Markov property,
\begin{align*}
P(Z(t)=j_0,\ldots,Z(t-k)=j_k| T> t)
&\sim  q_{k,j_k}{(f(1)^{j_0}-f(0)^{j_0})P_{j_k,j_{k-1}}\cdots P_{j_1,j_{0}}\over f(k,1)^{j_k}-f(k,0)^{j_k}}\\
&=q_{k,j_k}Q_{j_k,j_{k-1}}^{(k)}
\cdots Q_{j_{1},j_{0}}^{(1)}.
\end{align*}

Finally, observe that $(Q_{ij}^{(k)})_{j\geq 1}$ is a proper distribution with the  probability generating function
\begin{equation*}
\sum_{j=1}^\infty Q_{ij}^{(k)} s^j=\frac{f(sf(k-1,1))^i-f(sf(k-1,0))^i}{f(k,1)^i-f(k,0)^i}.
\end{equation*}

\subsection{Proof of Theorem \ref{LD}}
Recall notation $\bar R(s)=R'(s)/R(s)$ and observe that
\begin{equation*}
\overline{R}(s)={d\over ds}\ln{R(s)}=\sum_{j=0}^{\infty} \frac{1}{l^{j+1}}\frac{b'(f(j,s))f'(j,s)}{b(f(j,s))},
\end{equation*}
where $f'(j,s)={d\over ds}f(j,s)$. 
Put furthermore, $\bar{R}_{t}(s)=\frac{R_{t}'(s)}{R_{t}(s)}$ for $s\in [0,1]$ and $t\ge0$. Using \eqref{Rt}, we obtain
\begin{equation*}
\bar{R}_{t}(s)={d\over ds}\ln{R_{t}(s)}=\sum_{j=0}^{t-1} \frac{1}{l^{j+1}}\frac{b'(f(j,s))f'(j,s)}{b(f(j,s))}.
\end{equation*}

\begin{lem}\label{ld1}
 Assume $\gamma=0$, $f'(1)<\infty$, and put
 \begin{equation*}
\delta_t= \sum_{j=t}^\infty\gamma_0\cdots\gamma_{j-1},\quad \gamma_i=f'(f(i,1)).
\end{equation*}
Then $\delta_t\to0$ as $t\to\infty$ and
\begin{equation*}
\bar{R}(s)-\bar{R}_{t}(s)<\frac{f'(1)\delta_t}{p_l},\quad s\in [0,1].
\end{equation*}
\end{lem}
\begin{proof}
Using the expressions for $\bar{R}(s)$ and $\bar{R}_{t}(s)$, as well as the inequality $b(s)\ge1$, we see that indeed
 \begin{equation*}
\bar{R}(s)-\bar{R}_{t}(s)=\sum_{j=t}^\infty \frac{b'(f(j,s))f'(j,s)}{b(f(j,s))l^{j+1}}\le b'(1)\sum_{j=t}^\infty f'(j,1)< \frac{f'(1)\delta_t}{p_l}.
\end{equation*}
The fact that $\delta_t<\infty$ follows from $\gamma_i\to0$ as $ i\to\infty$, which, in turn, is a consequence of $\gamma=0$.
\end{proof}

\begin{lem}\label{ld2}
 Assume $f'(1)<\infty$, $\gamma=0$. The sequence \eqref{ck} is strictly decreasing.
 \end{lem}
\begin{proof}
It suffices to show that
\[1+f(s)\bar R(f(s)) <1+s\bar R(s),\quad s\in[0,1].\]
Using the definition of $R(\cdot)$ given in Proposition \ref{HR} it is easy to verify the equality
\[f(s)R(f(s)) = p_l(sR(s))^l,\]
which entails
\[\ln f(s)+\ln R(f(s)) = \ln p_l+l\ln s+l\ln R(s).\]
After differetiating
\[{f'(s)\over f(s)}+\bar R(f(s)) f'(s)= {l\over  s}+l\bar R(s),\]
we find
\[1+f(s)\bar R(f(s)) ={(\ln p_ls^l)'\over (\ln f(s))'} (1+s\bar R(s)),\]
where
${(\ln p_ls^l)'\over (\ln f(s))'} <1$, since
\[(\ln p_ls^l)'<(\ln p_ls^l)'+(\ln b(s))'=(\ln f(s))'.  \]
\end{proof}

\begin{lem}\label{ld3}
 If $\gamma=0$, then
\begin{align*}
 {f'(t,s)s\over f(t,s)}&=l^t(1+s\bar R_t(s)),\\
 {f''(t,s)s^2\over f(t,s)}&=l^{2t}(1+s\bar R_t(s))^2+l^t(s^2\bar R'_t(s)-1).
\end{align*}
 \end{lem}
\begin{proof}
Both relations are straightforward corollaries of formula \eqref{Rt}.
\end{proof}

Assuming $\gamma=0$, we first prove Theorem \ref{LD}-a using Lemmas \ref{ld1},  \ref{ld2} and \ref{ld3}, and then turn to the proof of Theorem \ref{LD}-b.

Let $f'(1)<\infty$.
From  \eqref{Tank2}, we compute the conditional expectation
\begin{equation*}
  E\left(Z(k)|T>t\right)=\frac{f'(k,f(t-k,1))f(t-k,1)}{f(k,f(t-k,1))},
\end{equation*}
and applying the first  relation in Lemma \ref{ld3}, we find
\begin{equation*}
  E\left(Y(k)|T>t\right)=1+f(t-k,1)\bar{R}_{k}(f(t-k,1)).
\end{equation*}
Thus the difference
\begin{align*}
 c(t-k)-E\left(Y(k)|T>t\right)=& f(t-k,1) (\bar{R}(f(t-k,1))-\bar{R}_{k}(f(t-k,1))
\end{align*}
is non-negative and bounded from above by a constant times $f(t-k,1) \delta_k$, see Lemma \ref{ld1}.
By monotonocity, we have for all $1\le k, k'\le t$,
\begin{align*}
 f(t-k,1) \delta_k
&\leq  f(t-k',1) \delta_0+\delta_{k'}.
\end{align*}
The obtained upper bound goes to 0 as first $t\to\infty$ and then $k'\to\infty$. This proves the uniform convergence stated in Theorem \ref{LD}-a.

Let $f''(1)<\infty$. To prove Theorem \ref{LD}-b it suffices to show the inequality
\begin{align*}
&Var(Y(k)|T>t)< c\,l^{-k}f(t-k,1),\quad 0\le k\le t,
\end{align*}
for some constant $c$.
Formula \eqref{Tank2} yields for $s=f(t-k,1)$,
\begin{equation*}
  Var\left(Z(k)|T>t\right)=\frac{f''(k,s)s^2}{f(k,s)}+\frac{f'(k,s)s}{f(k,s)}-\left(\frac{f'(k,s)s}{f(k,s)}\right)^2,
\end{equation*}
so that by Lemma \ref{ld3}, we get
\begin{equation*}
Var\left(Z(k)|T>t\right)=l^{k}f(t-k,1)\left(\bar{R}_{k}(f(t-k,1))+f(t-k,1) \bar{R}_{k}'(f(t-k,1))\right).
\end{equation*}

Since we already know that $\bar{R}_t(s)$ is uniformly bounded by a constant, it remains to establish a similar property for the derivative $\bar{R}'_t(s)$, which satisfies
\begin{eqnarray*}
\bar{R}_{t}'(s)<  \sum_{j=0}^\infty \frac{b''(f(j,s))f'(j,s)^2+b'(f(j,s))f''(j,s)}{l^{j+1} b(f(j,s))},
\end{eqnarray*}
 and since $b''(s)\leq f''(1)/p_l$, we obtain
 \begin{eqnarray*}
\bar{R}_{t}'(s)<\frac{f''(1)}{lp_l}\sum_{j=0}^\infty \frac{f'(j,1)^2}{l^j}+\frac{f'(1)}{lp_l}\sum_{j=0}^\infty \frac{f''(j,1)}{l^j}.
\end{eqnarray*}
We finish the proof by verifying that $\sum_{j=0}^\infty f''(j,1)<\infty$. Indeed, by the chain rule,
\begin{align*}
f''(j+1,1)&=\sum_{i=0}^{j} f'(i,1)^2 f''(f(i,1)) f'(f(i+1,1))\cdots f'(f(j,1))\\
&\le f''(1)\sum_{i=0}^{j} \gamma_0^2\cdots \gamma_{i-1}^2  \gamma_{i+1}\cdots \gamma_j,
\end{align*}
and because $\gamma_j\to0$ as $j\to\infty$, we have
\begin{align*}
\sum_{j=0}^\infty\sum_{i=0}^{j} \gamma_0^2\cdots \gamma_{i-1}^2  \gamma_{i+1}\cdots \gamma_j<\infty.
\end{align*}

\section{Proofs of Theorem \ref{extGW} and Propositions  \ref{peps}, \ref{peps2} and \ref{peps3}}\label{pr2}

For a sequence of defective GW-processes with reproduction laws $f_n(\cdot)$, we have
\begin{equation*}
P(T_n>t)=f_n(t,1)-f_n(t,0),
\end{equation*}
and by \eqref{Tank2},
\begin{equation}\label{Tank3}
E(e^{-\lambda Z_n(t-k)}| T_n>t)=\frac{f_n(t-k,e^{-\lambda} f_n(k,1))-f_n(t-k,e^{-\lambda} f_n(k,0))}{f_n(t,1)-f_n(t,0)},
\end{equation}
so that in particular,
\[E(e^{-\lambda Z_n(t)}| T_n>t)=\frac{f_n(t,e^{-\lambda})-f_n(t,0)}{f_n(t,1)-f_n(t,0)}.
\]

\subsection{Proof of Theorem \ref{extGW}}

Relation \eqref{An} is easily extended to the iterations of the generating functions
\[f_n(t,s)=r_n\hat f(t,s/r_n).\]
Therefore, if $\ln r_n\sim x/C(t_n)$, then
\[ f_n(t_n, e^{-\lambda /C(t_n)})=(1+o(1))\hat f(t_n, e^{-(\lambda+x+o(1)) /C(t_n)}),\quad n\to\infty.\]
On the other hand, by \eqref{psi} and
\[E(e^{-\lambda \hat Z(t)/C(t)}|\hat T_0>t)={\hat f(t, e^{-\lambda /C(t)})-\hat f(t, 0)\over1-\hat f(t, 0)},\]
we get
\[\hat f(t, e^{-\lambda /C(t)})\to\hat q+(1-\hat q)\Psi(\lambda),\quad t\to\infty.\]
This and the previous relation lead to the assertion of Theorem \ref{extGW}-a.

Turning to the proof of Theorem \ref{extGW}-b, observe that by   \eqref{phi},
\begin{equation*}
P(e^{- ub^t} \hat Z(t)<z|\hat T_0>t)\to \psi(u),\quad u\in(0,\infty),\quad z\in(0,\infty),
\end{equation*}
and therefore, for $\lambda\ge0$,
\[\hat f(t, e^{-\lambda e^{- ub^t}})\to\hat q+(1-\hat q)\psi(u),\quad t\to\infty,\]
implying
\begin{equation}\label{fu}
 \hat f(t, e^{-e^{- (u+o(1))b^t}})\to\hat q+(1-\hat q)\psi(u),\quad t\to\infty.
\end{equation}
If for some sequence $t_n\to\infty,$
\[\ln (1/r_n)=-e^{ -(y+o(1))b^{t_n}},\quad y\in(0,\infty),\quad n\to\infty,\]
then for fixed positive $\lambda$ and $u$, we can write
\[ f_n(t_n, e^{-\lambda  e^{- ub^{t_n}}})=(1+o(1))\hat f(t_n, \exp\{-  e^{- (u+o(1))b^{t_n}}-e^{ -(y+o(1))b^{t_n}}\}),\quad n\to\infty.\]
Applying \eqref{fu} we conclude that
\[ f_n(t_n, e^{-\lambda  e^{- ub^{t_n}}})\to\hat q+(1-\hat q)\psi(u\wedge y),\quad n\to\infty,\]
yielding
\begin{equation*}
P(e^{- ub^{t_n}}  Z(t_n)<z|T_n>t_n)\to {\psi(u\wedge y)\over\psi(y)},\quad u\in(0,\infty),\quad z\in(0,\infty),
\end{equation*}
and eventually for $u\in(0,y)$,
\[P(b^{-t_n}\ln Z_n(t_n)\le u| T_n>t_n)\to \psi(u)/\psi(y),\quad n\to\infty.\]

\subsection{Proof of Proposition \ref{peps}}

Here we deal with the sequence
\begin{equation}\label{master}
 f_n(t_n-k,s)=r_n-\left[\gamma_n^{t_n-k} (r_n-s)^{-\theta_n}+(1-\gamma_n^{t_n-k})(r_n-q_n)^{-\theta_n}\right]^{-1/\theta_n},
\end{equation}
assuming $\gamma_n\to\gamma\in(0,1)$, $\theta_n\to\theta\in(0,1]$, $q_n\to q\in [0,1)$, and $r_n\to1$ so that \eqref{aseps} holds. Note that  \eqref{aseps} implies $\gamma_n^{t_n}\to 0$.
Proposition \ref{peps}-a directly follows from two relations
\begin{align*}
 f_n(t_n,1)&=r_n-\left[\gamma_n^{t_n} (r_n-1)^{-\theta_n}+(1-\gamma_n^{t_n})(r_n-q_n)^{-\theta_n}\right]^{-1/\theta_n}\\
 &\to 1-(1-q)\left[1+(1-q)x^{-\theta}\right]^{-1/\theta},\\
 f_n(t_n,0)&=r_n-\left[\gamma_n^{t_n} r_n^{-\theta_n}+(1-\gamma_n^{t_n})(r_n-q_n)^{-\theta_n}\right]^{-1/\theta_n}\to q.
\end{align*}

Turning to Proposition \ref{peps}-b, let $k\geq 0$ and $t_n-k\to\infty$. In view of  \eqref{Tank3}, we have to show that putting $\hat\gamma_n=\gamma_n^{\frac{t_n-k}{\theta_n}}$,
\begin{align*}
& f_n(t_n-k,e^{-\lambda \hat\gamma_n}f_n(k,1))\to 1-(1-q)\left(1+(1-q)^{\theta}(\lambda+x)^{-\theta}\right)^{1/\theta},\\
 &f_n(t_n-k,e^{-\lambda \hat\gamma_n}f_n(k,0))\to q.
\end{align*}
The second convergence is easily obtained from \eqref{master} using
$$f_n(k,0)=r_n-(\gamma_n^{k} r_n^{-\theta_n}+(1-\gamma_n^{k})(r_n-q_n)^{-\theta_n})^{-1/\theta_n}\to 1-(\gamma^k+(1-\gamma^k)(1-q)^{-\theta})^{-1/\theta}.$$
The first convergence is also obtained from \eqref{master} using the following asymptotic formulas. Since $\gamma_n^{-k/\theta_n}(r_n-1)\to 0$, we have
\begin{align*}
1-f_n(k,1)\sim 1-r_n+ (\gamma_n^k (r_n-1)^{-\theta_n})^{-1/\theta_n}\sim  (r_n-1)(\gamma_n^{-k/\theta_n}-1).
\end{align*}
Thus
\begin{align*}
r_n-e^{-\lambda \hat\gamma_n}f_n(k,1)&\sim \lambda\gamma_n^{\frac{t_n-k}{\theta_n}}+(r_n-1)\gamma_n^{-k/\theta_n},
\end{align*}
implying
\begin{align*}
\gamma_n^{t_n-k}\big(r_n-e^{-\lambda \hat\gamma_n}f_n(k,1)\big)^{-\theta_n}&\sim \big(\lambda+(r_n-1)\gamma_n^{-\frac{t_n}{\theta_n}}\big)^{-\theta_n}\to (\lambda+x)^{-\theta}.
\end{align*}

\subsection{Proof of Proposition  \ref{peps2}}

Here we deal with the sequence
\[f_n(t,s)=r_n-(r_n-q_n)^{1-\gamma_n^{t}}(r_n-s)^{\gamma_n^t},\]
as $\gamma_n\to\gamma\in(0,1)$, $q_n\to q\in [0,1)$, and $r_n\to1$. We assume that \eqref{lug} holds for some $t_n\to\infty$.

Condition \eqref{lug}  gives
$$(r_n-1)^{\gamma_n^{t_n}}\to e^{-y},$$
which implies
\begin{align*}
 f_n(t_n,1)&= r_n-(r_n-q_n)^{1-\gamma_n^{t_n}}(r_n-1)^{\gamma_n^{t_n}}\to 1-(1-q)e^{-y},\\
 f_n(t_n,0)&= r_n-(r_n-q_n)^{1-\gamma_n^{t_n}}r_n^{\gamma_n^{t_n}}\to q.
\end{align*}
yielding Proposition  \ref{peps2}-a.

Let $k\geq 0$ and $t_n-k\to\infty$. To prove Proposition  \ref{peps2}-b it suffices to show that putting $\hat r_n=(r_n-1)^{uy^{-1}\gamma_n^{k}}$,
\begin{equation*}
E\left(e^{-\lambda \hat r_nZ_n(t_n-k)}| T_n>t_n\right)\to\frac{1-e^{-u}}{1-e^{-y}},\quad n\to\infty,
\end{equation*}
for $\lambda\geq 0$ and $u\in[0,y]$. This in turn, follows from
\begin{align*}
 f_n\left(t_n-k,e^{-\lambda\hat r_n}f_n(k,1)\right)&\to  1-(1-q)e^{-u},\\
f_n\left(t_n-k,e^{-\lambda\hat r_n}f_n(k,0)\right)&\to q,
\end{align*}
which we prove next.
The first of these two relations is obtained as follows: using
$$1-f_n(k,1)\sim (r_n-1)^{\gamma_n^k}(1-q)^{1-\gamma_n^k},$$
and taking into account that $ u\leq y$, we get
\begin{align*}
\left(r_n-e^{-\lambda\hat r_n}f_n(k,1)\right)^{\gamma_n^{t_n-k}}&\sim  \left(r_n-1+\lambda  \hat r_n+(r_n-1)^{\gamma_n^k}(1-q)^{1-\gamma_n^k}\right)^{\gamma_n^{t_n-k}}\\
&\sim  \left(\lambda \hat r_n\right)^{\gamma_n^{t_n-k}}\to e^{-u},
\end{align*}
and, as a consequence,
\begin{align*}
f_n\left(t_n-k,e^{-\lambda \hat r_n}f_n(k,1)\right)&=r_n-(r_n-q_n)^{1-\gamma_n^{t_n-k}}\left(r_n-e^{-\lambda \hat r_n}f_n(k,1)\right)^{\gamma_n^{t_n-k}}\\
&\to  1-(1-q)e^{-u}.
\end{align*}
The second relation follows from
\begin{equation*}
f_n(k,0)=r_n-r_n^{\gamma_n^k}(r_n-q_n)^{1-\gamma_n^k}\to 1-(1-q)^{1-\gamma^k}.
\end{equation*}

\subsection{Proof of Proposition  \ref{peps3}}

Here we deal with the sequence
\begin{align*}
f_n(t,s)=A_n-\left[\gamma_n^t (A_n-s)^{|\theta_n|}+(1-\gamma_n^t)(A_n-q_n)^{|\theta_n|}\right]^{1/|\theta_n|},
\end{align*}
as $\gamma_n\to\gamma\in (0,1)$, $q_n\to q\in [0,1)$, $A_n\to 1$, and $\theta_n\to 0$. We assume that \eqref{luga} holds for some $t_n\to\infty$.

Propositions  \ref{peps3}-a and \ref{peps3}-b$_2$ are proven similarly to Proposition  \ref{peps2}. To prove Proposition  \ref{peps3}-b$_1$, fix a $k\geq 0$ and let $t_n-k\to\infty$. We write $\hat u(x)={-x}\ln(1-u/x)$ and also
 $$\hat \theta_n=(1-uy^{-1}\gamma_n^k)^{y\gamma_n^{-t_n}}.$$
 It suffices to show that
\begin{equation*}
E\left(e^{-\lambda \hat \theta_nZ_n(t_n-k)}| T_n>t_n\right)\to\frac{1-e^{-u}}{1-e^{-y(1-e^{-a})}},\quad n\to\infty,
\end{equation*}
for $\lambda\geq 0$ and $u\in[0,y(1-e^{-a}))$, or in terms of generating functions,
\begin{align*}
 f_n\left(t_n-k,e^{-\lambda\hat \theta_n}f_n(k,1)\right)&\to 1- (1-q)e^{-u},\\
f_n\left(t_n-k,e^{-\lambda\hat \theta_n}f_n(k,0)\right)&\to q.
\end{align*}
We finish the proof by checking only the first of these two relations.

Since
$$A_n-f_n(k,1)=\left[(A_n-q_n)^{|\theta_n|}-\gamma_n^k\left(1-\left(A_n-1\right)^{|\theta_n|}\right)\right]^{1/|\theta_n|}=\left[1-\gamma^k(1-e^{-a})+o(1)\right]^{1/|\theta_n|},$$
we get
\begin{eqnarray*}
\left(A_n-e^{-\lambda\hat \theta_n}f_n(k,1)\right)^{|\theta_n|}=\left(\left[1-\gamma^k(1-e^{-a})+o(1)\right]^{1/|\theta_n|}+(\lambda+o(1))\hat \theta_n\right)^{|\theta_n|}.
\end{eqnarray*}
Using
\[\hat \theta_n^{\ |\theta_n|}\to 1-uy^{-1}\gamma^k,\]
and $ u< y(1-e^{-a})$, we obtain
\begin{align*}
f_n\left(t_n-k,e^{-\lambda \hat \theta_n}f_n(k,1)\right)&=1-(1-q)\left(1-(u/y+o(1))\gamma_n^{t_n}\right)^{1/ |\theta_n|}(1+o(1))\\
&\to  1-(1-q)e^{-u},
\end{align*}
since  $\left(1-\gamma_n^{t_n}\right)^{1/|\theta_n|}\to e^{-y}$ due to condition \eqref{luga}.

\section*{Acknowledgements}

This research is part of C. Minuesa's PhD project supported by the Spanish Ministerio de Educaci\'on, Cultura y Deporte (grants FPU13/03213 and EST15/00538), Junta de Extremadura (grant GR15105), Ministerio de Econom\'ia y Competitividad (grant MTM2015-70522-P), and the FEDER.




\end{document}